\RequirePackage{fix-cm}
\documentclass[smallextended]{svjour3}       % onecolumn (second format)
\smartqed  % flush right qed marks, e.g. at end of proof
\usepackage{graphicx}
\usepackage{amsmath}
\usepackage{amsfonts}
\usepackage{colonequals} %to use the := symbol
\usepackage{enumerate}

\usepackage{hyperref}                % For creating hyperlinks in cross references
\hypersetup{
    bookmarks=true,         % show bookmarks bar?
    unicode=false,          % non-Latin characters in Acrobat‚????s bookmarks
    pdftoolbar=true,        % show Acrobat‚????s toolbar?
    pdfmenubar=true,        % show Acrobat‚????s menu?
    pdffitwindow=false,     % window fit to page when opened
    pdfstartview={FitH},    % fits the width of the page to the window
    pdftitle={My title},    % title
    pdfauthor={Author},     % author
    pdfsubject={Subject},   % subject of the document
    pdfcreator={Creator},   % creator of the document
    pdfproducer={Producer}, % producer of the document
    pdfkeywords={keywords}, % list of keywords
    pdfnewwindow=true,      % links in new window
    colorlinks=true,       % false: boxed links; true: colored links
    linkcolor=blue,          % color of internal links
    citecolor=blue,        % color of links to bibliography
    filecolor=magenta,      % color of file links
    urlcolor=cyan           % color of external links
}

\begin{document}
\long\def\red#1{{\color{red}#1}}
\long\def\yel#1{{\color{yellow}#1}}
\long\def\mag#1{{\color{magenta}#1}}
\title{On the expected number of equilibria in a multi-player multi-strategy evolutionary game%\thanks{Grants or other notes
}

\titlerunning{Expected number of equilibria}        % if too long for running head

\author{Manh Hong Duong         \and
        The Anh Han %etc.
}


\institute{Manh Hong Duong \at
              Mathematics Institute, Zeeman Building, University of Warwick, Coventry CV4 7AL, UK. \\
%              Tel.: +123-45-678910\\
%              Fax: +123-45-678910\\
              \email{m.h.duong@warwick.ac.uk}           %  \\
           \and
           The Anh Han \at
               School of Computing, Teesside University, Borough Road, Middlesbrough, UK TS1 3BA \\
               \email{ T.Han@tees.ac.uk}
}

\date{Received: date / Accepted: date}

\maketitle

\begin{abstract}
In this paper, we analyse the mean number $E(n,d)$ of internal equilibria in a general $d$-player $n$-strategy evolutionary game where the agents' payoffs are normally distributed. First, we give a computationally implementable formula for the general case. Next, we characterize the asymptotic behaviour of $E(2,d)$, estimating its lower and upper bounds as $d$ increases. Then we provide a closed formula for $E(n,2)$. Two important consequences are obtained from this analysis. On the one hand, we show that in both cases the probability of seeing the maximal possible number of equilibria tends to zero when $d$ or $n$ respectively goes to infinity. On the other hand, we demonstrate that the expected number of stable equilibria is bounded within a certain interval. Finally, for larger $n$ and $d$, numerical results are provided and discussed.   \\

%\textbf{Keywords: } Multi-player game theory, multiple strategies, random polynomials, number of equilibria, random games. 

%Insert your abstract here. Include keywords, PACS and mathematical
%subject classification numbers as needed.
\end{abstract}

\section{Introduction}
Evolutionary game theory is the suitable mathematical framework whenever there is frequency dependent selection -- the fitness of an individual does not only depend on its strategy, but also on the composition of the population in relation with (multiple) other strategies  \cite{maynard-smith:1982to,hofbauer:1998mm,nowak:2006bo}. 
The payoff from the games is interpreted as individual fitness, naturally leading  to a dynamical approach. 
As in classical game theory with the  Nash equilibrium \cite{mclennan2005asymptotic,mclennan2005expected},  
the analysis of properties of equilibrium points in evolutionary game theory has been of special interest \cite{maynard-smith:1982to,broom:1997aa,gokhale:2010pn}. 
Herein, equilibrium points of a dynamical system predict the composition of strategy frequencies where all the strategies have the same average fitness. Biologically, these points can predict a co-existence of different types in a population and the maintenance of polymorphism. 

Recently, some attention has been paid to  both numerical and analytical studies  of equilibrium points and their stability in random evolutionary games \cite{gokhale:2010pn,HTG12}. %\red{Hong: should we move this paragraph as the first sentences of the 4th paragraph below??}
 The focus was  on analyzing the probability of observing a certain number of equilibria (whether counting all equilibria or only the stable ones) if the payoff entries are randomly drawn. 
 This probability allows one to predict the complexity of the interactions as the number of strategies and the number of players in the game increase, especially when the environments are unknown or changing rapidly over time \cite{fudenberg:1992bv,gross:2009aa}. Furthermore, these studies have paid substantial attention to the maximal number of equilibrium points and the attainability of the patterns of evolutionarily stable strategies, as knowing them is insightful, and historically they have been studied extensively,  not only in classical and evolutionary game theory, but also in other fields such as population genetics \cite{maynard-smith:1982to,karlin:1980aa,vickers:1988aa,vickers:1988ab,karlin:1970tp,vickers:1988ac,haigh1989large,broom1994sequential,broom:1993pa,broom:1997aa,altenberg:2010tp,gokhale:2010pn,HTG12,broom2013game}. However, as the studies deal with the concrete numbers of equilibrium points, they have needed to take a direct approach that consists of solving a system of polynomial equations, the degree of which increases with the number of players in a game. As such, the mathematical analysis was mostly restricted to evolutionary games with a small number of players, due to the impossibility of solving general polynomial equations of a high degree \cite{able:1824aa,HTG12}.

%\red{In the present work we have  focused exclusively on deriving the expected number of equilibria in general, without checking whether they are stable, i.e. being an evolutionary stable state (ESS) \cite{maynard-smith:1982to,hofbauer:1998mm,haigh1989large}. However, as a consequence of  Theorem \ref{thm: estimate for E(d)} and the observation  that an equilibrium in a random two-strategy game (with arbitrary number of players) is stable with probability $1/2$ \cite[Theorem 3]{HTG12}, we can easily derive that the expected number of stable equilibrium points in random games with $d$ players and two strategies is bounded within $\left[\frac{d-1}{2\pi\sqrt{2d-3}}, \ \ \frac{1}{2\pi}\sqrt{d-1}\sqrt{1+\frac{\pi}{2}\sqrt{d-1}}\right]$. 

%\red{The questions regarding ESS in random games, for example, the expected number of ESSs for two-player games, have been studied, see \cite{haigh1988distribution}. Also, patterns of ESSs in an evolutionary games have been considered in \cite{haigh1989large,broom1994sequential}. 

%\red{Should we move this into the discussion section. Probably we should write this paper as a totally dependent paper with the question we asked. in the discussion we will compare with previous papers: HTG12 and other papers on random polynomial....
In this paper, we ask instead the  question:  what is the mean or expected number of equilibria that one can observe if the payoff matrix entries of the game are randomly drawn? 
Knowing the mean number of equilibria not only gives important insights into the overall complexity of the interactions as the number of participating players in the game and the potential strategies that the players can adopt are magnified. It also enables us to predict the boundaries of the concrete numbers of (stable) equilibrium points such as the maximal one as we show later on in the paper. 
By connecting to  the theory of random polynomials \cite{EK95}, we first provide an exact, computationally implementable, formula for the expected number of equilibria in a general multi-player multi-strategy random game  when the payoff entries are  normally  distributed. Secondly, we derive lower and upper bounds of such a formula for the case of two-player games and provide an explicit formula for the case of two-strategy games. As a consequence, we can derive similar bounds when considering only equilibrium points that are stable.
Finally, numerical results are provided and discussed when there are more players and strategies. 

The rest of the paper is structured as follows. In Section~\ref{sec: M2M} we introduce the models and methods: the replicator equation in evolution game theory and the random polynomial theory are summarized in Sections~\ref{sec: game theory} and~\ref{sec: RP theory}, respectively. The link between them, which is a major aspect of our method,  is described in Section~\ref{sec: Link}. The main results of this paper are presented in Section~\ref{sec: Results}, starting with two-strategy games in Section~\ref{sec: two-stategy}, then with two-player games in Section~\ref{sec: two-player}, and lastly, with the general case of games with arbitrary numbers of players and strategies in Section \ref{sec: 2multi}. We compare our results with related ones in the literature and discuss  some future directions in Section~\ref{sec: discussion}. Finally, some detailed computations are given in the Appendix.

\section{Models and Methods}
\label{sec: M2M}
\subsection{Evolutionary game theory and replicator dynamics}
\label{sec: game theory}
The classical approach to evolutionary games is replicator dynamics  \cite{taylor:1978wv,zeeman:1980ze,hofbauer:1998mm,schuster:1983le,nowak:2006bo}, describing that whenever a strategy has a fitness larger than the  average fitness of the population, it is expected to  spread. Formally, let us consider an infinitely large population with $n$ strategies, numerated from 1 to $n$. They have frequencies $x_i$, $1 \leq i \leq n$, respectively, satisfying that $0 \leq x_i \leq 1$  and $\sum_{i = 1}^{n} x_i = 1$. The interaction of the individuals in the population is in randomly selected groups of $d$ participants, that is, they play and obtain their fitness from   $d$-player games.  We consider here symmetrical games (e.g. the public goods games and their generalizations \cite{hardin:1968mm,hauert:2002in,santos:2008xr,pacheco:2009aa,Han:2014tl}) in which the order of the participants is  irrelevant. Let $\alpha^{i_0}_{i_1,\ldots,i_{d-1}}$ be the payoff of the focal player, where $i_0$ ($1 \leq i_0 \leq n$) is the strategy of the focal player, and let $i_k$ (with $1 \leq i_k \leq n$ and $1 \leq k \leq d-1$) be the strategy of the player in position $k$. 
These payoffs form a $(d-1)$-dimensional payoff matrix \cite{gokhale:2010pn}, which satisfies  (because of the game symmetry)  
\begin{equation}
\label{eq:symmetry}
\alpha^{i_0}_{i_1,\ldots,i_{d-1}} = \alpha^{i_0}_{i_1',\ldots,i_{d-1}'},
\end{equation}
 whenever $\{i_1'\ldots,i_{d-1}'\}$ is a permutation of  $\{i_1\ldots,i_{d-1}\}$. This means that  only the fraction of each strategy in the game matters. 
 
The average payoff or fitness of the focal player is given by  

\begin{equation}
\pi_{i_0} = \sum\limits_{i_1,\ldots,i_{d-1}=1}^{n}\alpha^{i_0}_{i_1,\ldots,i_{d-1}}\prod\limits_{k=1}^{d-1}x_{i_k}.
\end{equation}
%Considering symmetrical games where $\alpha_{i_0,i_1\ldots,i_{d-1}} = \alpha_{i_0',i_1'\ldots,i_{d-1}'}$ whenever $\{i_1'\ldots,i_{d-1}'\}$ is a permeation of  $\{i_1\ldots,i_{d-1}\}$ (this means that the order of the players in the game is irrelevant: only the fraction of each strategy in the game matters). 
By abuse of notation, let us denote $\alpha^{i_0}_{k_{1}, ..., k_{n} } := \alpha^{i_0}_{i_1,\ldots,i_{d-1}}$, where $k_i$, $1 \leq i \leq n$, with $\sum^{n}_{i = 1}k_i = d-1$,   is the number of players using strategy $i$ in $\{i_1,\ldots,i_{d-1}\}$. 
Hence, from Equation \eqref{eq:symmetry}, the fitness  of strategy $i_0$ can be rewritten  as follows 
\begin{equation}
\label{eq: eqn for fitness}
\pi_{i_0}  = \sum\limits_{\substack{0\leq k_{1}, ..., k_{n}\leq d-1,\\ \sum\limits^{n}_{i = 1}k_i = d-1  }}\alpha^{i_0}_{k_{1}, ..., k_{n} }\begin{pmatrix}
d-1\\
k_{1}, ..., k_n
\end{pmatrix} \prod\limits_{i=1}^{n}x_{i}^{k_i} \quad \text{ for } i_0 = 1, \dots, n,
\end{equation}
where 
$\begin{pmatrix}
d-1\\
k_{1}, \dots, k_{n}
\end{pmatrix} = \frac{(d-1)!}{\prod\limits _{k=1}^n k_i! }$
are the multinomial coefficients. 

Now  the replicator equations for games with $n$ strategies can be written as follows  \cite{hofbauer:1998mm,sigmund:2010bo} 
\begin{equation} 
\label{eq:replicator_eqs_2playersGame}
\dot{x_i} = x_i \left(\pi_i - \langle \pi \rangle \right) \quad\quad \text{ for } i = 1,\dots,n-1,
\end{equation} 
where $\langle \pi \rangle = \sum^n_{k = 1} x_k \, \pi_k$ is the average payoff of the population. The equilibrium points of the system are given by the points $(x_1, \dots, x_n)$ satisfying the condition that the  fitness of all strategies are the same. That is, they are represented by solutions of  the system of equations 
\begin{equation} 
\label{eq:system_fitness}
\pi_i = \pi_n \quad \text { for all } 1 \leq i \leq n-1.
\end{equation} 
%\red{Our approach is also applicable for mixed Nash equilibria ...} 

Subtracting from each of the equations the term $\pi_n$  we obtain  a system of $n-1$ polynomials of degree $d-1$
\begin{equation}
\label{eq: eqn for fitness2}
\sum\limits_{\substack{0\leq k_{1}, ..., k_{n}\leq d-1,\\ \sum^{n}\limits_{i = 1}k_i = d-1  }}\beta^i_{k_{1}, ...,k_{n-1} }\begin{pmatrix}
d-1\\
k_{1}, ..., k_n
\end{pmatrix} \prod\limits_{i=1}^{n}x_{i}^{k_i} = 0\quad \text{ for } i = 1, \dots, n-1,
\end{equation}
where $\beta^{i}_{k_{1}, ..., k_{n-1} } := \alpha^{i}_{k_{1}, ..., k_{n} } -\alpha^{n}_{k_{1}, ..., k_{n} } $. Assuming that all the payoff entries have the same probability distribution, then all $\beta^{i}_{k_{1}, ..., k_{n-1} }$, $i = 1, \dots, n-1$,  have symmetric distributions, i.e. with mean 0 (see also the proof in \cite{HTG12}).

In the following analysis, we focus on   internal equilibrium points \cite{gokhale:2010pn,HTG12}, i.e. $0 < x_i < 1$ for all $1 \leq i \leq n-1$. Hence, by using the transformation $y_i = \frac{x_i}{x_n}$, with $0< y_i < +\infty$ and $1 \leq i \leq n-1$, dividing the left hand side of the above equation by $x_n^{d-1}$ we obtain the following equation in terms of $(y_1,\ldots,y_{n-1})$ that is equivalent to~\eqref{eq: eqn for fitness2} 
\begin{equation}
\label{eq: eqn for fitnessy}
 \sum\limits_{\substack{0\leq k_{1}, ..., k_{n-1}\leq d-1,\\  \sum\limits^{n-1}_{i = 1}k_i \leq d-1  }}\beta^i_{k_{1}, ..., k_{n-1} }\begin{pmatrix}
d-1\\
k_{1}, ..., k_{n}
\end{pmatrix} \prod\limits_{i=1}^{n-1}y_{i}^{k_i} = 0\quad  \quad \text{ for } i = 1, \dots, n-1.
\end{equation}
As stated, one of the main goals of this article is to compute the expected number of (internal) equilibria in a general  $n$-strategy $d$-player random evolutionary game. That consists in  computing  the expected  number of solutions $(y_1, \dots, y_{n-1}) \in {\mathbb{R}_+}^{n-1}$ of the system of $(n-1)$ polynomials of degree $(d-1)$ in (\ref{eq: eqn for fitnessy}). Furthermore, herein our analysis focuses on  payoff matrices  with  normally distributed entries. It is known that, even for $n=2$, it is impossible to analytically solve the system whenever $d > 5$ \cite{able:1824aa}, as seen in \cite{HTG12}. Hence,  it is not feasible to use this direct approach of analytically solving the system   if one wants to deal with   games with a large number of players and with multiple strategies. In this work, we address this issue by connecting to the theory of random polynomials described in the following section.   
\subsection{Random polynomial theory}
\label{sec: RP theory}
Keeping the form of Eq.~\eqref{eq: eqn for fitnessy} in mind, we consider a system of $n-1$ random polynomials of degree $d-1$,
\begin{equation}
\label{eq: eqn for systems}
 \sum\limits_{\substack{0\leq k_{1}, ..., k_{n-1}\leq d-1,\\  \sum^{n-1}_{i = 1}k_i \leq d-1  }}a^i_{k_{1}, ..., k_{n-1} }\prod\limits_{i=1}^{n-1}y_{i}^{k_i} = 0\quad  \quad \text{ for } i = 1, \dots, n-1,
\end{equation}
where $a^i_{k_{1}, ..., k_{n-1} }$ are independent and identically  distributed (i.i.d.) multivariate normal random vectors with mean zero and covariance matrix $C$. Denote by $v(y)$
the vector whose components are all the monomials $\big\{\prod\limits_{i=1}^{n-1}y_i^{k_i}\big\}$ where $0\leq k_i\leq d-1$ and $\sum\limits_{i=1}^{n-1}k_i\leq d-1$. Let $A$ denote the random matrix whose $i-$th row contains all coefficients $a^i_{k_1,\ldots, k_{n-1}}$. Then~\eqref{eq: eqn for systems} can be re-written as
\begin{equation}
\label{eq: eqn for system 2}
A v(y)=0.
\end{equation}
The following theorem is the starting point of the analysis of this paper.
\begin{theorem}\cite[Theorem 7.1]{EK95}
\label{theo: EK95}
Let $U$ be any measurable subset of $\mathbb{R}^{n-1}$. Assume that the rows of $A$ are i.i.d. multivariate normal random vectors with mean zero and covariance matrix $C$. The expected number of real roots of the system of equations~\eqref{eq: eqn for system 2} that lie in the set $U$ is given by
\[
\pi^{-\frac{n}{2}}\Gamma\left(\frac{n}{2}\right)\int_U\left(\det\left[\frac{\partial^2}{\partial x_i\partial y_j}(\log v(x)^T C v(y))\big|_{y=x=t}\right]_{ij}\right)^\frac{1}{2}\, dt.
\]
\end{theorem}
%In the sequel we will apply this theorem to~\eqref{eq: eqn for systems}, i.e., with
%\[
%a^i_{k_1,\ldots, k_{n-1}}=\beta^i_{k_{1}, ..., k_{n-1} }\begin{pmatrix}
%d-1\\
%k_{1}, ..., k_n
%\end{pmatrix}.
%\]
%\textbf{Question: What is the expected number of equilibria that a $d$-player with $n$-strategies has?}
\subsection{From random polynomial theory to evolutionary game theory} 
\label{sec: Link}
Let $E(n,d)$ be the number of internal equilibria in a $d-$player random game with $n$ strategies. As has been shown in Section~\ref{sec: game theory}, $E(n,d)$ is the same as the number of positive solutions of Eq.~\eqref{eq: eqn for fitnessy}. 
We will apply Theorem~\ref{theo: EK95} with $a^i_{k_{1}, ..., k_{n-1} }=\beta^i_{k_{1}, ..., k_{n-1} }\begin{pmatrix}
d-1\\
k_{1}, ..., k_{n}
\end{pmatrix}$ and $U=[0,\infty)^{n-1}\subset \mathbb{R}^{n-1}$.

Suppose that all $\beta^{i}_{k_{1}, ..., k_{n-1} }$ are  Gaussian distributions with mean $0$ and variance $1$, then for each $i$ ($1 \leq i \leq n-1$), $A^i=\left\{\beta^i_{k_{1}, ..., k_{n-1} }\begin{pmatrix}
d-1\\
k_{1}, ..., k_{n}
\end{pmatrix}\right\}
$ is a multivariate normal random vector with mean zero and covariance matrix $C$ given by
\begin{equation}
\label{eq: matrix C}
C=\mathrm{diag} \left(\begin{pmatrix}
d-1\\
k_{1}, ..., k_{n}
\end{pmatrix}^2\right)_{0\leq k_i\leq d-1,~\sum\limits_{i=1}^{n-1}k_i\leq d-1}.
\end{equation}
We obtain the following lemma, which is a direct application of Theorem \ref{theo: EK95}.
\begin{lemma}
\label{lemma: E(n,d)}
Assume that $\{A^i\}_{1\leq i\leq n-1}$ are independent normal random vectors with mean zero and covariance matrix $C$ as in \eqref{eq: matrix C}. The expected number of internal equilibria in a \emph{d}-player \emph{n-}strategy random game  is given by 
\begin{equation}
\label{eq: expected number or EQ}
E(n,d)=\pi^{-\frac{n}{2}}\Gamma\left(\frac{n}{2}\right)\underbrace{\int_0^\infty\ldots\int_0^\infty}_{n-1\,\text{times}}\left(\det\left[\frac{\partial^2}{\partial x_iy_j}(\log v(x)^T C v(y))\big|_{y=x=t}
\right]_{ij}\right)^\frac{1}{2}\,dt,
\end{equation}
where
\begin{equation}
v(x)^T C v(y)=\sum\limits_{\substack{0\leq k_{1}, ..., k_{n-1}\leq d-1,\\ \sum\limits^{n-1}_{i = 1}k_i\leq d-1  }}\begin{pmatrix}
d-1\\
k_1,\ldots,k_{n}
\end{pmatrix}^2\prod\limits_{i=1}^nx_i^{k_i}y_i^{k_i}.
\end{equation}
\end{lemma}
Denote by $L$ the matrix with entries
\begin{align*}
L_{ij}&=\frac{\partial^2}{\partial x_iy_j}(\log v(x)^T C v(y))\big|_{y=x=t},
\end{align*}
then $E(n,d)$ can be written as
\begin{equation}
\label{eq: expected number or EQ 2}
E(n,d)=\pi^{-\frac{n}{2}}\Gamma\left(\frac{n}{2}\right)\underbrace{\int_0^\infty\ldots\int_0^\infty}_{n-1\,\text{times}}\left(\det L\right)^\frac{1}{2}dt.
\end{equation}
It has been shown that a $d$-player $n$-strategy game has at most $(d-1)^{n-1}$ isolated internal equilibria (and this bound is sharp) \cite{HTG12}. We denote by $p_i$, $1 \leq i \leq (d-1)^{n-1}$, the probability that the game has exactly $i$ such equilibria. Then $E(n,d)$ can also be defined through $p_i$ as follows 
\begin{equation}
\label{eq: definition E}
E(n,d)=\sum_{i=1}^{(d-1)^{n-1}}i \cdot p_i. 
\end{equation}
\section{Results}
\label{sec: Results}
We start with the case where there are two strategies ($n = 2$), analytically deriving the upper and lower bounds for $E(2, d)$. Next we derive exact results for games with two players ($d = 2$). 
Finally, we provide numerical results and discussion for the general case with an arbitrary number of players and strategies. For ease of representation, we  start by assuming that the coefficients $\beta^i_{k_1,\ldots,k_{n-1}}$ are standard normal distributions. We then show that the results do not change if they have arbitrary identical normal distributions. 
%\footnote{For simplicity, we only consider normal distributions with mean zero since normal distributions with mean non-zero can also be shifted to the former case.}. 
% 

%two normal distributions, not  normal distributions themselves]   
% with mean zero and arbitrary variance
\subsection{Multi-player two-strategy games}
\label{sec: two-stategy}
We first consider games with an arbitrary number of players, but having only two strategies, i.e. $n = 2$. In this case,   Equation (\ref{eq: eqn for fitnessy}) is simplified to the following  univariate polynomial equation of degree $d-1$ with $y \in \mathbb{R}_+$
\begin{equation}
\label{eq: eqn for y}
\sum\limits_{k=0}^{d-1}\beta_k\begin{pmatrix}
d-1\\
k
\end{pmatrix}y^k=0.
\end{equation}
%\subsubsection{Independent and identically distributions}
The following lemma describes a closed form  of $E(2,d)$.
\begin{lemma}
\label{lemma: density of zeros}
Assume that $\beta_k$ are independent Gaussian distributions with variance $1$ and mean $0$. Then the number of internal equilibria, $E(2,d)$, in a $d-$player random game with two strategies is given by
\begin{equation}
\label{eq: E(d)}
E(2,d)=\int_0^\infty f(t)\,dt,
\end{equation}
where
\begin{equation}
\label{eq: density of zeros}
f(t)=\frac{1}{\pi}\left[\frac{\sum\limits_{k=1}^{d-1}k^2\begin{pmatrix}
d-1\\
k
\end{pmatrix}^2t^{2(k-1)}}{\sum\limits_{k=0}^{d-1}\begin{pmatrix}
d-1\\
k
\end{pmatrix}^2t^{2k}}-\left(\frac{\sum\limits_{k=1}^{d-1}k\begin{pmatrix}
d-1\\
k
\end{pmatrix}^2t^{2k-1}}{\sum\limits_{k=0}^{d-1}\begin{pmatrix}
d-1\\
k
\end{pmatrix}^2t^{2k}}\right)^2\right]^{\frac{1}{2}}.
\end{equation}
\end{lemma}
\begin{proof}
Since $\beta_k$ has Gaussian distribution with variance $1$ and mean $0$, $\beta_k\begin{pmatrix}
d-1\\
k
\end{pmatrix}$ has Gaussian distribution with variance $\begin{pmatrix}
d-1\\
k
\end{pmatrix}^2$ and mean $0$. According to Lemma~\ref{lemma: E(n,d)}, the equality \eqref{eq: E(d)} holds with
\begin{equation}
\label{eq: f(t)}
f(t)=\frac{1}{\pi}\left[\frac{\partial^2}{\partial x\partial y}\left(\log v(x)^T C v(y)\right)\Big|_{y=x=t}\right]^\frac{1}{2},
\end{equation}
where the vector $v$ and the matrix $C$ (covariance matrix) are given by
\begin{equation}
v(x)=\begin{pmatrix}
1\\
x\\
\vdots\\
x^{d-1}
\end{pmatrix}, \quad C=(C_{ij})_{i,j=1,\ldots,d-1}\quad\text{with}~~ C_{ij}=\delta_{ij}\begin{pmatrix} 
d-1\\
i
\end{pmatrix}\begin{pmatrix}
d-1\\
j
\end{pmatrix},
\end{equation}
where $\delta_{ij}$ is the Kronecker notation,
\begin{equation*}
\delta_{ij}=\begin{cases}
1\quad\textrm{if}\quad i=j,\\
0\quad\textrm{if}\quad i\neq j.
\end{cases}
\end{equation*}
A straightforward calculation gives
\[
v(x)^T C v(y)=\sum\limits_{k=0}^{d-1}\begin{pmatrix}
d-1\\
k
\end{pmatrix}^2x^ky^k,
\]
and
\begin{align*}
\frac{\partial^2}{\partial x\partial y}\left(\log v(x)^T C v(y)\right)&=\frac{\sum\limits_{k=1}^{d-1}k^2\begin{pmatrix}
d-1\\
k
\end{pmatrix}^2x^{k-1}y^{k-1}}{\sum\limits_{k=0}^{d-1}\begin{pmatrix}
d-1\\
k
\end{pmatrix}^2x^{k}y^{k}}
\\& \ \ - \frac{\left(\sum\limits_{k=1}^{d-1}k\begin{pmatrix}
d-1\\
k
\end{pmatrix}^2x^ky^{k-1}\right)\left(\sum\limits_{k=1}^{d-1}k\begin{pmatrix}
d-1\\
k
\end{pmatrix}^2x^{k-1}y^{k}\right)}{\left(\sum\limits_{k=0}^{d-1}\begin{pmatrix}
d-1\\
k
\end{pmatrix}^2x^{k}y^{k}\right)^2}.
\end{align*}
Substituting this expression into~\eqref{eq: f(t)}, we obtain~\eqref{eq: density of zeros}.
%\red{Hong: need double check the calculation}
\end{proof}
\begin{example}
For the cases  $d=2$ and $d=3$, we have
\begin{align*}
&E(2,2)=\frac{1}{\pi}\int_0^\infty \frac{1}{1+t^2}\,dt = \frac{1}{\pi} \lim_{t \rightarrow +\infty} \tan^{-1}(t) = 0.5\label{eq: f(t) for d=3},
\\& E(2,3)=\frac{2}{\pi}\int_0^\infty \frac{\sqrt{t^4+t^2+1}}{t^4+4t^2+1}\,dt\approx 0.77.
\end{align*}
\end{example}
%\red{These results are in accordance with the numerical results from \cite{gokhale:2010pn}, noting the definition (\ref{definition:Ed}) of $E(n,d)$ . }
%\begin{remark}
%The results do not change if $\beta_k$ have any normal distribution with mean 0 and arbitrary variance $\sigma^2$. In this case, $\beta_k\begin{pmatrix}
%d-1\\
%k
%\end{pmatrix}$ has Gaussian distribution with variance $\sigma^2\begin{pmatrix}
%d-1\\
%k
%\end{pmatrix}^2$ and mean $0$. Using the same procedure  we obtain the same $f(t)$ function as in Lemma~\ref{lemma: density of zeros}, which is not dependent on $\sigma$. 
%This result  suggests that when dealing with random games as in this article, it is sufficient to consider that payoff entries are from the interval $[0, 1]$ instead off from an arbitrary one.  These are  in accordance  with the results in \cite{HTG12} for the computation with small $d$. 
%\end{remark}
%
The following proposition presents some properties of the density function $f(t)$, which will be  used later for estimating the asymptotic behaviour of $E(2,d)$.
\begin{proposition}
The following properties hold
\begin{enumerate}[1)]
\item \begin{equation}
\label{eq: f(t)2}
f(t)=\frac{d-1}{\pi}\frac{\sqrt{\sum\limits_{k=0}^{2d-4}a_k\,t^{2k}}}{{\sum\limits_{k=0}^{d-1}\begin{pmatrix}
d-1\\
k
\end{pmatrix}^2t^{2k}}},
\end{equation}
where
\begin{align}
& a_k=\sum\limits_{\substack{1\leq i\leq d-1\\ 1\leq j\leq d-2\\i+j=k}}\begin{pmatrix}
d-1\\
i
\end{pmatrix}^2\begin{pmatrix}
d-2\\
j
\end{pmatrix}^2-\sum\limits_{\substack{1\leq i'\leq d-2\\ 1\leq j'\leq d-2\\i'+j'=k}}\begin{pmatrix}
d-2\\
i'
\end{pmatrix}\begin{pmatrix}
d-1\\
i'+1
\end{pmatrix}\begin{pmatrix}
d-2\\
j'
\end{pmatrix}\begin{pmatrix}
d-1\\
j'+1
\end{pmatrix},\label{eq: formula of a}
\\
&a_{k}=a_{2d-4-k}, \quad \text{for all}~~0\leq k\leq 2d-4,
\quad a_{0}=a_{2d-4}=1,\label{eq: properties of a}
\quad a_k\geq 1.
\end{align}
\item $f(0)=\frac{d-1}{\pi},\quad f(1)=\frac{d-1}{2\pi}\frac{1}{\sqrt{2d-3}}$.
\item $f(t)=\frac{1}{2\pi}\left[\frac{1}{t}G'(t)\right]^\frac{1}{2}$,
where
\[
G(t)=t\frac{\frac{d}{dt}M_d(t)}{M_d(t)}=t\frac{d}{dt}\log M_d(t).
\]
\item $t\mapsto f(t)$ is a decreasing function.
\item $f\left(\frac{1}{t}\right)=t^2 f(t)$.
\item \begin{equation}
\label{eq: integral f from 0 to 1}
E(2,d)=2\int_0^1\, f(t)dt=2\int_1^\infty f(t)\,dt.
\end{equation}
\end{enumerate}
\end{proposition}
\begin{proof}
\begin{enumerate}[1)]
\item Set 
\begin{align}
\label{eq: A, M}
& M_d(t)=\sum\limits_{k=0}^{d-1}\begin{pmatrix}
d-1\\
k
\end{pmatrix}^2t^{2k},\quad 
A_d(t)=\sum\limits_{k=1}^{d-1}k^2\begin{pmatrix}
d-1\\
k
\end{pmatrix}^2t^{2(k-1)},
\\ &B_d(t)=\sum\limits_{k=1}^{d-1}k\begin{pmatrix}
d-1\\
k
\end{pmatrix}^2t^{2k-1}\label{eq: B}.
\end{align}
Then 
\[
f(t)=\frac{1}{\pi}\frac{\sqrt{A_d(t)M_d(t)-B_d(t)^2}}{M_d(t)}.
\]
Using
\[
k\begin{pmatrix}
d-1\\
k
\end{pmatrix}=(d-1)\begin{pmatrix}
d-2\\
k-1
\end{pmatrix},
\]
we can transform 
\begin{align*}
&A_d(t)=(d-1)^2\sum\limits_{k=1}^{d-1}\begin{pmatrix}
d-2\\
k-1
\end{pmatrix}^2t^{2(k-1)}=(d-1)^2\sum\limits_{k=0}^{d-2}\begin{pmatrix}
d-2\\
k
\end{pmatrix}^2t^{2k}=(d-1)^2M_{d-1}(t),
\\
&B_d(t)=(d-1)\sum\limits_{k=1}^{d-1}\begin{pmatrix}
d-2\\
k-1
\end{pmatrix}\begin{pmatrix}
d-1\\
k
\end{pmatrix}t^{2k-1}=(d-1)t\sum\limits_{k=0}^{d-2}\begin{pmatrix}
d-2\\
k
\end{pmatrix}\begin{pmatrix}
d-1\\
k+1
\end{pmatrix}t^{2k}.
\end{align*}
Therefore,
\begin{align*}
&A_d(t)M_d(t)-B_d(t)^2
\\&\quad =(d-1)^2\left[\left(\sum\limits_{k=0}^{d-2}\begin{pmatrix}
d-2\\
k
\end{pmatrix}^2t^{2k}\right)\left(\sum\limits_{k=0}^{d-1}\begin{pmatrix}
d-1\\
k
\end{pmatrix}^2t^{2k}\right)-t^2\left(\sum\limits_{k=0}^{d-2}\begin{pmatrix}
d-2\\
k
\end{pmatrix}\begin{pmatrix}
d-1\\
k+1
\end{pmatrix}t^{2k}\right)^2\right]
\\ &\quad=(d-1)^2\sum\limits_{k=0}^{2d-4}a_k\,t^{2k},
\end{align*}
where
 \begin{align*}
a_k =\sum\limits_{\substack{0\leq i\leq d-1\\ 0\leq j\leq d-2\\i+j=k}}\begin{pmatrix}
d-1\\
i
\end{pmatrix}^2\begin{pmatrix}
d-2\\
j
\end{pmatrix}^2-\sum\limits_{\substack{0\leq i'\leq d-2\\ 0\leq j'\leq d-2\\i'+j'=k-1}}\begin{pmatrix}
d-2\\
i'
\end{pmatrix}\begin{pmatrix}
d-1\\
i'+1
\end{pmatrix}\begin{pmatrix}
d-2\\
j'
\end{pmatrix}\begin{pmatrix}
d-1\\
j'+1
\end{pmatrix}. 
 \end{align*}
For the detailed computations of $a_k$ and the proof of~\eqref{eq: properties of a}, see Appendix~\ref{sec: properties of a}.
\item The value of $f(0)$ is found directly from~\eqref{eq: f(t)2}. For the detailed computations of $f(1)$, see Appendix~\ref{sec: f(1)}. 
\item It follows from~\eqref{eq: A, M}-\eqref{eq: B} that  
\begin{equation}
\label{eq: A and B interms of M}
B_d(t)=\frac{1}{2}M_d'(t),\quad A_d(t)=\frac{1}{4 t}(tM_d'(t))',
\end{equation}
where $'$ is derivative with respect to $t$. Hence
\begin{align*}
f(t)&=\frac{1}{2\pi}\left(\frac{\frac{1}{t}(tM_d'(t))'M_d(t)-M_d'(t)^2}{M_d(t)^2}\right)^\frac{1}{2}
\\&=\frac{1}{2\pi}\left(\frac{(tM_d'(t))'M_d(t)-tM_d'(t)^2}{tM_d(t)^2}\right)^\frac{1}{2}
\\&=\frac{1}{2\pi}\left(\frac{1}{t}G'(t)\right)^\frac{1}{2},
\end{align*}
where $G(t)=t\frac{M_d'(t)}{M_d(t)}$.
\item Since $M_d(t)$ contains only even powers of $t$ with positive coefficients, all of its roots are purely imaginary. Suppose that
\[
M_d(t)=\prod_{i=1}^{d-1} (t^2+r_i),
\]
where $r_i>0$ for all $1\leq i\leq d-1$. It follows that
\[
G(t)=t\frac{M_d'(t)}{M_d(t)}=\sum\limits_{i=1}^{d-1}\frac{2t^2}{t^2+r_i},
\]
and hence
\begin{equation}
\label{eq: fourth property}
(2\pi f(t))^2=\frac{1}{t}G'(t)=\sum\limits_{i=1}^{d-1}\frac{4r_i}{(t^2+r_i)^2}.
\end{equation}
Since $r_i>0$ for all $i=1,\ldots,d-1$, the above equality implies that $f(t)$ is decreasing in $t\in [0,\infty)$.
\item Set
\[
g(t):=\sqrt{A_d(t)M_d(t)-B_d(t)^2},
\]
Then 
\begin{equation}
\label{eq: f interms of A, B, M}
f(t)=\frac{1}{\pi}\frac{g(t)}{M_d(t)}.
\end{equation}
It follows from the symmetric properties of the binomial coefficients that
\[
M_d\left(\frac{1}{t}\right)=\frac{1}{t^{2(d-1)}}M_d(t).
\]
Similarly, from~\eqref{eq: properties of a} we have
\[
g\left(\frac{1}{t}\right)=\frac{1}{t^{2(d-2)}}g(t).
\] 
Therefore
\[
f\left(\frac{1}{t}\right)=\frac{1}{\pi}\frac{g(1/t)}{M_d(1/t)}=\frac{1}{\pi} t^2\frac{g(t)}{M_d(t)}=t^2 f(t).
\]

%Taking derivative of both sides, we obtain
%\[
%8\pi^2 f(t)f'(t)=\sum\limits_{i=1}^{d-1}\frac{-4(a_i-t^2)}{(a_i+t^2)^3}
%\]
\item By change of variable, $s=\frac{1}{t}$, and from 5), we have
\[
\int_1^\infty f(t)\,dt=\int_1^0 f\left(\frac{1}{s}\right)\frac{-1}{s^2}\,ds=\int_0^1 f(s)\, ds.
\]
Therefore
\[
E(2,d)=\int_0^\infty f(t)\,dt=\int_0^1 f(t)\,dt+\int_1^\infty f(t)\,dt=2\int_0^1 f(t)\,dt.
\]
\end{enumerate}
\end{proof}
\begin{remark}
We provide an alternative proof of the fifth property in the above lemma in Appendix~\ref{sec: alternative proof}.
\end{remark}
\begin{remark}
Besides enabling a significantly less complex numerical computation of $E(2,d)$ (see already our numerical results using this formula in Table  \ref{table:Ed}), the equality~\eqref{eq: integral f from 0 to 1} reveals an interesting property: the expected number of zeros of the polynomial $P(y)$ in two intervals $(0,1)$ and $(1,\infty]$ are the same. Equivalently, the expected numbers of internal equilibria in two intervals $(0,\frac{1}{2}]$ and $(\frac{1}{2},1)$ are equal since $y=\frac{x}{1-x}$. Indeed, this result also confirms the observation that by swapping indices between the two strategies ($n = 2$) we move equilibrium from $x$ to  $1 -x$, thereby not resulting in change in the number of equilibria.  
\end{remark}

Based on the analytical formula of $E(2,d)$, we now provide upper and lower bounds for the mean number of equilibria as the payoff entries of the game are randomly drawn. 
\begin{theorem}
\label{thm: estimate for E(d)}
$E(2,d)$ satisfies the following estimate
\begin{equation}
\label{eq: estimate for E(d)}
\frac{d-1}{\pi\sqrt{2d-3}}\leq E(2,d)\leq\frac{1}{\pi}\sqrt{d-1}\sqrt{1+\frac{\pi}{2}\sqrt{d-1}}.
\end{equation}
%\red{Consequently, we have that $E(d) = O\left(n^{3/4}\right)$ as $d\rightarrow +\infty$. }
\end{theorem}
\begin{proof}
Since $f(t)$ is decreasing, we have $f(0)\geq f(t)\geq f(1)$ for $t\in[0,1]$. As a consequence,
\[
E(2,d)\geq 2\int_0^1 f(1)\,dt=2f(1)=\frac{d-1}{\pi\sqrt{2d-3}}.
\]
To obtain the upper bound, we proceed as follows.
\begin{align}
\label{eq: E(d)^2}
E(2,d)^2&=4\left(\int_0^1 f(t)\,dt\right)^2\notag
\\&\leq 4\int_0^1 f(t)^2\,dt \quad\text{(by Jensen's inequality)}\notag
\\&=\frac{1}{\pi^2}\int_0^1\frac{1}{t}G'(t)\,dt\notag
\\&\stackrel{\eqref{eq: fourth property}}{=}\frac{1}{\pi^2}\int_0^1 \sum\limits_{i=1}^{d-1}\frac{4r_i}{(t^2+r_i)^2}\,dt\notag
\\&=\frac{1}{\pi^2}\sum\limits_{i=1}^{d-1}\int_0^1\frac{4r_i}{(t^2+r_i)^2}\,dt\notag
%\\&\leq \frac{1}{\pi^2}\sum\limits_{i=1}^{d-1}\int_0^1 \frac{4}{r_i}\,dt\quad 
%\\&=\frac{4}{\pi^2}\sum\limits_{i=1}^{d-1}\frac{1}{r_i}
%\\&=\frac{4}{\pi^2}\frac{\sum\limits_{i=1}^{d-1}\prod\limits_{j\neq i} r_j}{\prod\limits_{i=1}^{d-1} r_i}
%\\&=\frac{4}{\pi^2} (d-1).
\\&=\frac{1}{\pi^2}\sum\limits_{i=1}^{d-1}\left(\frac{2}{r_i+1}+\frac{\cot^{-1}(\sqrt{r_i})}{\sqrt{r_i}}\right).
\end{align}
%Note that in the second inequality above we have used the fact that
%\[
%\frac{4r_i}{(t^2+r_i)^2}\leq \frac{4}{r_i}, \quad \text{for}~~ 0\leq t\leq 1,
%\]
%and in the last equality, we have use the Vieta's theorem.
%Hence, $E_d\leq \frac{2}{\pi}\sqrt{d-1}$.
Note that if $ab=1$, then
\begin{equation}
\label{eq: simple equality}
\frac{1}{a+1}+\frac{1}{b+1}=1.
\end{equation}
We observe that if $z$ is a zero of $M_d(t)$, then $\frac{1}{z}$ is also a zero because $M_d(t)=t^{2(d-1)}M_d(1/t)$. This implies that the sequence $\{r_i, i=1,\ldots,d-1\}$ can be grouped into $\frac{d-1}{2}$ pairs of the form $\left(a,\frac{1}{a}\right)$. Using~\eqref{eq: simple equality}, we obtain 
\begin{equation}
\label{eq: the first}
\frac{1}{\pi^2}\sum\limits_{i=1}^{d-1}\frac{2}{r_i+1}=
\frac{1}{\pi^2}(d-1).
\end{equation}

For the second term, since $\cot^{-1}(z)\leq \frac{\pi}{2}~~ \text{for all}~~ z\geq 0$, we have
\begin{align}
\label{eq: the second}
\sum\limits_{i=1}^{d-1}\frac{\cot^{-1}(\sqrt{r_i})}{\sqrt{r_i}}&\leq\frac{\pi}{2}\sum\limits_{i=1}^{d-1}\frac{1}{\sqrt{r_i}}\notag
\\&=\frac{\pi}{2}\frac{\sum\limits_{i=1}^{d-1}\prod\limits_{j\neq i}\sqrt{r_j}}{\prod\limits_{i=1}^{d-1}\sqrt{r_i}}\notag
\\&=\frac{\pi}{2}\sum\limits_{i=1}^{d-1}\prod\limits_{j\neq i}\sqrt{r_j}\notag
\\&\leq \frac{\pi}{2}\sqrt{(d-1)\sum\limits_{i=1}^{d-1}\prod\limits_{j\neq i}r_j}\notag
\\&=\frac{\pi}{2}(d-1)^\frac{3}{2},
\end{align}
where we have used the Cauchy-Schwartz inequality
\[
\left(\sum\limits_{i=1}^nb_i\right)^2\leq n\sum\limits_{i=1}^nb_i^2,
\]
and the fact that $\prod\limits_{i=1}^{d-1}r_i=1 $ and $\sum\limits_{i=1}^{d-1}\prod\limits_{j\neq i}r_j=(d-1)^2$ according to  Vieta's theorem for the roots $\{r_i\}$ of $M_d$.

From \eqref{eq: E(d)^2}, \eqref{eq: the first} and~\eqref{eq: the second}, we have
\[
E(2,d)^2\leq \frac{1}{\pi^2}\left((d-1)+\frac{\pi}{2}(d-1)^\frac{3}{2}\right)=\frac{1}{\pi^2}(d-1)(1+\frac{\pi}{2}\sqrt{d-1}),
\]
or equivalently
\[
E(2,d)\leq \frac{1}{\pi}\sqrt{d-1}\sqrt{1+\frac{\pi}{2}\sqrt{d-1}}.
\]

\end{proof}
In Figure \ref{fig:Ed}a, we show the numerical results for $E(2,d)$ in comparison with the obtained upper and lower bounds.
\begin{corollary}
\begin{enumerate}[1)]
\item The expected number of equilibria increases unboundedly when $d$ tends to infinity
\begin{equation}
\label{eq:Edinfinity}
\lim_{d\rightarrow\infty}E(2,d)=+\infty.
\end{equation}
\item The probability $p_m$ of observing $m$ equilibria, $1\leq m\leq d-1$, is bounded by
\begin{equation}
\label{eq: bounded probablity of observing equilibria}
p_m\leq \frac{E(2,d)}{m}\leq\frac{1}{\pi m}\sqrt{d-1}\sqrt{1+\frac{\pi}{2}\sqrt{d-1}}.
\end{equation}
In particular,
\[
p_{d-1}\leq \frac{1}{\pi}\frac{\sqrt{1+\frac{\pi}{2}\sqrt{d-1}}}{\sqrt{d-1}},\quad\text{and}\quad
\lim_{d\rightarrow\infty}p_{d-1}=0.
\]
\end{enumerate}
\end{corollary}
\begin{proof}
\begin{enumerate}[1)]
\item This is a direct consequence of~\eqref{eq: estimate for E(d)}, as the lower bound of $E(2,d)$ tends to infinity when   $d$ tends to infinity.
\item This is again a direct consequence of~\eqref{eq: estimate for E(d)} and definition of $E(2,d)$. For any $1\leq m\leq d-1$, we have
\[
E(2,d)=\sum_{i=1}^{d-1}p_i\cdot i\geq p_m\cdot m.
\]
In particular,
\[
p_{d-1}\leq \frac{E(2,d)}{d-1}\leq \frac{1}{\pi}\frac{\sqrt{1+\frac{\pi}{2}\sqrt{d-1}}}{\sqrt{d-1}}.
\]
As a consequence, $\lim\limits_{d\rightarrow\infty} p_{d-1}=0$. More generally, we can see that this limit is  true for $p_{k}$ for any $k=O(d)$ as $d\rightarrow\infty$.
\end{enumerate}
\end{proof}
From this corollary we can see that, interestingly, although the mean number of equilibria tends to infinity when the number of players $d$ increases, the probability to see the maximal number of equilibria in a $d$-player system converges  to 0.  There has been extensive research studying the maximal number of equilibrium points of a dynamical system \cite{maynard-smith:1982to,karlin:1980aa,vickers:1988aa,vickers:1988ab,karlin:1970tp,vickers:1988ac,broom:1993pa,broom:1997aa,altenberg:2010tp}.
Our results suggest that the possibility to reach such a maximal number is very small when $d$ is sufficiently large. 

\begin{corollary}
The expected number of stable equilibrium points in a random game with $d$ players and two strategies is equal to $ \frac{E(2,d)}{2}$, and  is thus bounded within  to the following interval 
$$\left[\frac{d-1}{2\pi\sqrt{2d-3}}, \  \frac{1}{2\pi}\sqrt{d-1}\sqrt{1+\frac{\pi}{2}\sqrt{d-1}}\right].$$
\end{corollary}
\begin{proof} 
From \cite[Theorem 3]{HTG12}, it is known that  an equilibrium in a random game with two strategies and an arbitrary number of players, is stable with probability $1/2$. Hence, the corollary is  a direct consequence of  Theorem \ref{thm: estimate for E(d)}. 
\end{proof}

%\red{In the present work we have  focused exclusively on deriving the expected number of equilibria in general, without checking whether they are stable, i.e. being an evolutionary stable state (ESS) \cite{maynard-smith:1982to,hofbauer:1998mm,haigh1989large}. However, as a consequence of  Theorem \ref{thm: estimate for E(d)} and the observation  that an equilibrium in a random two-strategy game (with arbitrary number of players) is stable with probability $1/2$ \cite[Theorem 3]{HTG12}, we can easily derive that the expected number of stable equilibrium points in random games with $d$ players and two strategies is bounded within $\left[\frac{d-1}{2\pi\sqrt{2d-3}}, \ \ \frac{1}{2\pi}\sqrt{d-1}\sqrt{1+\frac{\pi}{2}\sqrt{d-1}}\right]$. 

\subsection{Two-player multi-strategy games}
\label{sec: two-player}
In this section, we consider games with two players, i.e. $d=2$, and arbitrary strategies. In this case~\eqref{eq: eqn for fitnessy} is simplified to a linear system
\begin{equation}
\label{eq: two player}
\begin{cases}
\sum_{j}\beta^i_j \,y_j=0, \quad\text{for}~~i=1,\ldots,n-1,\\
\sum_j y_j=1,
\end{cases}
\end{equation}
where $\beta_j^i$ have Gaussian distributions with mean $0$ and variance $1$. The main result of this section is the following explicit formula for $E(n,2)$.
\begin{theorem} 
\label{theo: E(n,2)}
We have
\begin{equation}
E(n,2)= \frac{1}{2^{n-1}}.
\end{equation} 
\end{theorem}
This theorem can be seen as a direct application of \cite[Theorem 1]{HTG12} and the definition of $E(n,d)$ in Eq. \eqref{eq: definition E}. Indeed, according to \cite[Theorem 1]{HTG12}, the probability that Eq. \eqref{eq: two player} has a (unique) isolated solution is $p_1=2^{1-n}$ and the probability of having no  solution is $0$. By \eqref{eq: definition E}, $E(n,2)=p_1$, and hence $E(n,2)=2^{1-n}$. \\
% \end{proof}

To demonstrate further the usefulness of our approach using the random polynomial theory, we present here an alternative proof that can directly obtain $E(n,2)$, i.e. without computing the probability of observing concrete numbers of equilibria. Our proof requires an assumption that $\{A^i=(\beta_j^i)_{j=1,\cdots,n-1}; i=1,\ldots, n-1\}$ are independent random vectors. We will elaborate and discuss this assumption in Remark \ref{rem: assumption}. The main ingredient of our proof is the following lemma, whose proof is given in Appendix~\ref{sec: computeDetL}.
\begin{lemma} Let $t_1,\dots,t_{n-1}$ be real numbers, and let $L$ be the matrix with entries
\begin{align*}
\\& L_{ii}=\frac{1}{1+\sum\limits_{k=1}^{n-1} t_k^2}-\frac{t_i^2}{(1+\sum\limits_{k=1}^{n-1} t_k^2)^2},
\\& L_{ij}=\frac{-t_i t_j}{(1+\sum\limits_{k=1}^{n-1} t_k^2)^2}\quad  \forall i \neq j. 
\end{align*}
It holds that
\label{lem: det L}
 \begin{equation}
\det L = \frac{1}{\left(1+\sum\limits_{k=1}^{n-1} t_k^2\right)^n}.
\end{equation}
\end{lemma}

%Eq. \eqref{eq: two player} is a linear e.quation $A y=b$, where
%
%\begin{equation*}
%A=\begin{pmatrix}
%\beta^1_1 & \ldots&\beta^1_{n-1}\\
% \ldots&\ldots&\ldots\\
%\beta^{n-1}_1&\ldots&\beta^{n-1}_{n-1}\\
%1&1&1
%\end{pmatrix}\,\quad b=\begin{pmatrix}
%0\\
%\vdots\\
%1
%\end{pmatrix}.
%\end{equation*}
%By Cramer’s theorem, if $\det A\neq 0$ then Eq \eqref{eq: two player} has a unique solution and if $\det A=0$ it has either no solution or infinitely many solutions. Let $p_1$ be the probability that $\det A\neq 0$. According to \cite[Theorem 1]{HTG12}, $p_1=2^{1-n}$. By definition of $E(n,d)$ in \eqref{eq: definition E}, we have $E(n,2)$
%\begin{equation}
%E(n,2)=p_1=2^{1-n}
%\end{equation}
%
%We use the following result.
%\begin{theorem}
%In a random two-player game with n strategies, 
%\end{theorem}
%
%\end{proof}}

\begin{proof}[Proof of Theorem \ref{theo: E(n,2)} under the assumption that  $\{A^i=(\beta_j^i)_{j=1,\cdots,n-1}; i=1,\ldots, n-1\}$ are independent random vectors.]

Let $L$ be the matrix in Lemma \ref{lem: det L}. According to Lemma~\ref{lemma: E(n,d)} and Lemma \ref{lem: det L}, we have
\begin{align*}
E(n,2)&=\pi^{-\frac{n}{2}}\Gamma\left(\frac{n}{2}\right)\underbrace{\int_0^\infty\ldots\int_0^\infty}_{n-1\,\text{times}}  \left(\det L \right)^{1/2} \,dt_1\ldots dt_{n-1}\\
&=\pi^{-\frac{n}{2}}\Gamma\left(\frac{n}{2}\right)\underbrace{\int_0^\infty\ldots\int_0^\infty}_{n-1\,\text{times}}  \left(1+\sum_{k=1}^{n-1} t_k^2\right)^{-n/2} \,dt_1\ldots dt_{n-1} \\
&=\frac{1}{2}\pi^{-\frac{n-1}{2}}\Gamma\left(\frac{n-1}{2}\right)\underbrace{\int_0^\infty\ldots\int_0^\infty}_{n-2\,\text{times}}  \left(1+\sum_{k=1}^{n-2} t_k^2\right)^{-(n-1)/2} \,dt_1\ldots dt_{n-2} \\
&\ldots \\
&=\frac{1}{2^{n-1}}\pi^{-\frac{1}{2}}\Gamma\left(\frac{1}{2}\right) \\
&=\frac{1}{2^{n-1}},
%\\&=\pi^{-\frac{n}{2}}\Gamma\left(\frac{n}{2}\right)\int_0^\infty\ldots\int_0^\infty \frac{1}{(1+\sum t_k^2)^{n-1}}\prod \sqrt{1-t_i^2}\,dt_1\ldots dt_{n-1}
\end{align*}
where we have repeatedly used the equality  (with $a  > 0$ and $p > 1$) $$\int_0^\infty (a + t^2)^{-p} dt = \frac{\sqrt{\pi} \ \Gamma{(p-\frac{1}{2})}}{2  \ \Gamma(p)}a^{\frac{1}{2}-p}.$$
\end{proof}
\begin{remark}
Note that, as in \cite[Theorem 2]{HTG12}, it is known that the probability that a two-player $n$-strategy  random game has a stable (internal) equilibrium is at most $2^{-n}$. Hence, the expected number of stable equilibrium points in this case is at most $2^{-n}$. From \cite[Theorem 2]{HTG12}, we also can show  that this bound is sharp only when $n = 2$. 
\end{remark}
\begin{remark}
\label{rem: assumption}
Let $M$ be the number of sequences $(k_1,\ldots, k_{n-1})$ such that $\sum_{j=1}^{n-1}k_j\leq d-1$. For each $i=1,\ldots, n-1$, we define a random vector  $A^i=(\beta^i_{k_1,\ldots,k_{n-1}})\in \mathbb{R}^M$. Recalling from Section \ref{sec: M2M} that $\beta^{i}_{k_{1}, ..., k_{n-1} }= \alpha^{i}_{k_{1}, ..., k_{n} } -\alpha^{n}_{k_{1}, ..., k_{n} } $, where  $\alpha^{i}_{i_1,\ldots,i_{d-1}}$ is the payoff of the focal player, with  $i$ ($1 \leq i \leq n$) being the strategy of the focal player, and $i_k$ (with $1 \leq i_k \leq n$ and $1 \leq k \leq d-1$) is the strategy of the player in position $k$.  The assumption in Lemma \ref{lemma: E(n,d)} is that $A^1,\ldots, A^{n-1}$ are independent random vectors and each of them has mean $0$ and covariance matrix $C$. This assumption clearly holds for $n = 2$. For $n>2$, the assumption  holds only under quite restrictive conditions such as $\alpha^{n}_{k_{1}, ..., k_{n} }$ is deterministic or $\alpha^{i}_{k_{1}, ..., k_{n} }$ are essentially identical. Nevertheless, from a purely mathematical point of view (see for instance \cite{Kos93,Kos01} and Section \ref{sec: discussion} for further discussion), it is important to investigate the number of real zeros of the system~\eqref{eq: eqn for fitnessy} under the assumption of independence of $\{A^i\}$ since this system has not been studied in the literature. As such, the investigation provides new insights into the theory of zeros of systems of random polynomials. In addition, it also gives important hint on the complexity of the game theoretical question, i.e. the number of expected number of equilibria. Indeed, for $d = 2$ and arbitrary $n$, the assumption does not affect the final result as shown in Theorem \ref{theo: E(n,2)}. We recall that \cite[Theorem 7.1]{EK95} also required that the rows $\{A^i\}$ are independent random vectors. Hence to treat the general case, one would need to generalize \cite[Theorem 7.1]{EK95} and this is a difficult problem as demonstrated in \cite{Kos93,Kos01}. With this motivation, in the next section, we  keep the assumption of independence of $\{A^i\}$ and leave the general case for future research. We numerically compute the number of zeros of the system \eqref{eq: eqn for fitnessy} for $n\in \{2,3,4\}$. Again, as we shall see, the numerical results under the given assumption lead to outcomes closely  corroborated by previous works in \cite{gokhale:2010pn,HTG12}.
\end{remark}
\subsection{Multi-player multi-strategy games}
\label{sec: 2multi}
We now move to the general case of multi-player games with multiple strategies. We provide numerical results for this case.
For simplicity of notation in this section we write $\sum\limits_{k_1,\ldots,k_{n-1}}$ instead of  $\sum\limits_{\substack{0\leq k_{1}, ..., k_{n-1}\leq d-1,\\ \sum^{n-1}_{i = 1}k_i\leq d-1  }}$ and $\sum\limits_{k_1,\ldots,k_{n-1}  | k_{i_1}, .., k_{i_m}  }$ instead of  $\sum\limits_{\substack{0\leq k_{i}\leq d-1 \  \forall i \in S \setminus R, \\ 1\leq k_j\leq d-1 \ \forall j \in R , \\ \sum^{n-1}\limits_{i = 1}k_i\leq d-1  }}$ where $S =  \{1, \ldots, n-1\}$ and $R =  \{i_1, \ldots, i_m\} \subseteq S$.

According to Lemma~\ref{lemma: E(n,d)}, the expected number of internal equilibria in a $d-$player random game with $n$ strategies is given by
\begin{align*}
E(n,d)&=\pi^{-\frac{n}{2}}\Gamma\left(\frac{n}{2}\right)\int_U\left(\det\left[\frac{\partial^2}{\partial x_iy_j}(\log v(x)^T C v(y))\big|_{y=x=t}
\right]_{ij}\right)^\frac{1}{2}dt
\\&=\pi^{-\frac{n}{2}}\Gamma\left(\frac{n}{2}\right)\int_U\left(\det L\right)^\frac{1}{2}dt,
\end{align*}
where $\Gamma$ is the Gamma function, and $L$ denotes the matrix with entries
\begin{align*}
L_{ij}&=\frac{\partial^2}{\partial x_iy_j}(\log v(x)^T C v(y))\big|_{y=x=t}.
\end{align*}
We have
\[
v(x)^T C v(y)=\sum\limits_{k_1,\ldots,k_{n-1}}\begin{pmatrix}
d-1\\
k_1,\ldots,k_{n}
\end{pmatrix}^2\prod\limits_{i=1}^nx_i^{k_i}y_i^{k_i}.
\]

Set $\Pi(x,y):=\prod\limits_{l=1}^{n-1}x_{l}^{k_l}y_{l}^{k_l}$. Then
\begin{align*}
&\frac{\partial^2}{\partial x_iy_j}(\log v(x)^T C v(y))
\\&\quad=\frac{\frac{1}{x_iy_j}\sum\limits_{k_{1}, ..., k_{n-1} | k_i, k_j}k_ik_j\begin{pmatrix}
d-1\\
k_1,\ldots,k_{n-1}
\end{pmatrix}^2\Pi(x,y)}{\sum\limits_{k_1,\ldots,k_{n-1}}\begin{pmatrix}
d-1\\
k_1,\ldots,k_{n-1}
\end{pmatrix}^2\Pi(x,y)}
\\&\qquad-\frac{\left(\frac{1}{x_i}\sum\limits_{k_{1},\ldots, k_{n-1} | k_i}k_i\begin{pmatrix}
d-1\\
k_1,\ldots,k_{n}
\end{pmatrix}^2\Pi(x,y)\right)\left(\frac{1}{y_j}\sum\limits_{k_{1}, ..., k_{n} |k_j}k_j\begin{pmatrix}
d-1\\
k_1,\ldots,k_{n}
\end{pmatrix}^2\Pi(x,y)\right)}{
\left(\sum\limits_{k_1,\ldots,k_{n-1}}\begin{pmatrix}
d-1\\
k_1,\ldots,k_{n}
\end{pmatrix}^2\Pi(x,y)\right)^2}.
\end{align*}
Therefore%\red{check: the index in the enumerator run from 1, in the denominator run from 0?? [TA: That is correct, I guess we need to modify a bit to take that into account; see also Mathematica file]}

\begin{align}
\label{eq:Lij}
L_{ij}&=\frac{\frac{1}{t_it_j}\sum\limits_{k_1,\ldots,k_{n-1} | k_i, k_j}k_ik_j\begin{pmatrix}
d-1\\
k_1,\ldots,k_{n}
\end{pmatrix}^2\prod\limits_lt_l^{2k_l}}{\sum\limits_{k_1,\ldots,k_{n-1}}\begin{pmatrix}
d-1\\
k_1,\ldots,k_{n}
\end{pmatrix}^2\prod\limits_lt_l^{2k_l}}
\\&\qquad-\frac{\left(\frac{1}{t_i}\sum\limits_{k_{1}, ..., k_{n-1} | k_i}k_i\begin{pmatrix}
d-1\\
k_1,\ldots,k_{n}
\end{pmatrix}^2\prod\limits_lt_l^{2k_l}\right)\left(\frac{1}{t_j}\sum\limits_{k_{1}, ..., k_{n-1} | k_j}k_j\begin{pmatrix}
d-1\\
k_1,\ldots,k_{n}
\end{pmatrix}^2\prod\limits_lt_l^{2k_l}\right)}{
\left(\sum\limits_{k_1,\ldots,k_{n-1}}\begin{pmatrix}
d-1\\
k_1,\ldots,k_{n}
\end{pmatrix}^2\prod\limits_lt_l^{2k_l}\right)^2}.
\end{align}
So far in the paper we have assumed that
all the $\beta^i_{k_{1}, ..., k_{n-1} }$ in Eq. \eqref{eq: eqn for fitnessy} are standard normal distributions. The following lemma shows, as a  consequence of the above described formula, that all the results obtained so far remain valid if they have a normal distribution with mean zero and arbitrary variance (i.e.  the  entries of the game payoff matrix have  the same, arbitrary normal distribution). 
\begin{lemma}
Suppose  $\beta^i_{k_{1}, ..., k_{n-1} }$ have normal distributions with mean 0 and arbitrary variance $\sigma^2$. Then, $L_{ij}$ as defined in (\ref{eq:Lij}) does not depend on $\sigma^2$.
\end{lemma}
\begin{proof}
In this case, $\beta^i_{k_{1}, ..., k_{n-1} }\begin{pmatrix}
d-1\\
k
\end{pmatrix}$ has a Gaussian distribution with variance $\sigma^2\begin{pmatrix}
d-1\\
k
\end{pmatrix}^2$ and mean $0$. Hence, 
\abovedisplayskip=0pt
\belowdisplayskip=0pt
\begin{equation*}
v(x)^T C v(y)=\sum\limits_{k_1,\ldots,k_{n-1}}\sigma^2\begin{pmatrix}
d-1\\
k_1,\ldots,k_{n}
\end{pmatrix}^2\prod\limits_{i=1}^nx_i^{k_i}y_i^{k_i}.
\end{equation*}

Repeating the same calculation we obtain the same $L_{ij}$ as in \eqref{eq:Lij}, which is independent of $\sigma$. 
\end{proof}
This result suggests that when dealing with random games as in this article, it is sufficient to consider that payoff entries are from the interval $[0, 1]$ instead of from an arbitrary one, as done numerically in \cite{gokhale:2010pn}. A similar behaviour has been observed in \cite{HTG12} for the analysis and computation with small $d$ or $n$, showing that results are not dependent on the interval where the payoff entries are drawn. 
Furthermore, it is noteworthy that this result confirms the observation that multiplying a payoff matrix by a constant does not change the equilibria \cite{hofbauer:1998mm}, provided  a one-to-one mapping from [0,1] to the arbitrary interval. 
\begin{example}[d-players with n=3 strategies]
\begin{align*}
L_{11}&=\frac{\frac{1}{t_1^2}\sum\limits_{k_1=1,k_2=0}k_1^2\begin{pmatrix}
d-1\\
k_1,k_2
\end{pmatrix}^2t_1^{2k_1}t_2^{2k_2}}{\sum\limits_{k_1,k_2}\begin{pmatrix}
d-1\\
k_1,k_2
\end{pmatrix}^2t_1^{2k_1}t_2^{2k_2}}-\frac{\left(\frac{1}{t_1}\sum\limits_{k_{1}=1, k_2=0}k_1\begin{pmatrix}
d-1\\
k_1,k_2
\end{pmatrix}^2t_1^{2k_1}t_2^{2k_2}\right)^2}{
\left(\sum\limits_{k_1,k_2}\begin{pmatrix}
d-1\\
k_1,k_2
\end{pmatrix}^2t_1^{2k_1}t_2^{2k_2}\right)^2},\\
L_{12}&=\frac{\frac{1}{t_1t_2}\sum\limits_{k_1=1,k_2=1}k_1k_2\begin{pmatrix}
d-1\\
k_1,k_2
\end{pmatrix}^2t_1^{2k_1}t_2^{2k_2}}{\sum\limits_{k_1,k_2}\begin{pmatrix}
d-1\\
k_1,k_2
\end{pmatrix}^2t_1^{2k_1}t_2^{2k_2}}
\\&\qquad-\frac{\left(\frac{1}{t_1}\sum\limits_{k_{1}=1, k_2=0}k_1\begin{pmatrix}
d-1\\
k_1,k_2
\end{pmatrix}^2t_1^{2k_1}t_2^{2k_2}\right)\left(\frac{1}{t_2}\sum\limits_{k_{1}=0, k_2=1}k_2\begin{pmatrix}
d-1\\
k_1,k_2
\end{pmatrix}^2t_1^{2k_1}t_2^{2k_2}\right)}{
\left(\sum\limits_{k_1,k_2}\begin{pmatrix}
d-1\\
k_1,k_2
\end{pmatrix}^2t_1^{2k_1}t_2^{2k_2}\right)^2},
\\L_{21}&=L_{12},
\\L_{22}&=\frac{\frac{1}{t_2^2}\sum\limits_{k_1=0,k_2=1}k_2^2\begin{pmatrix}
d-1\\
k_1,k_2
\end{pmatrix}^2t_1^{2k_1}t_2^{2k_2}}{\sum\limits_{k_1,k_2}\begin{pmatrix}
d-1\\
k_1,k_2
\end{pmatrix}^2t_1^{2k_1}t_2^{2k_2}}-\frac{\left(\frac{1}{t_2}\sum\limits_{k_{1}=0, k_2=1}k_2\begin{pmatrix}
d-1\\
k_1,k_2
\end{pmatrix}^2t_1^{2k_1}t_2^{2k_2}\right)^2}{
\left(\sum\limits_{k_1,k_2}\begin{pmatrix}
d-1\\
k_1,k_2
\end{pmatrix}^2t_1^{2k_1}t_2^{2k_2}\right)^2}.
\end{align*}
Therefore,
\[
E(3,d)=\pi^{-\frac{3}{2}}\Gamma\left(\frac{3}{2}\right)\int_0^\infty\int_0^\infty\sqrt{L_{11}L_{22}-L_{12}L_{21}}\,dt_1dt_2.
\]
\end{example} 
Next we provide some numerical results. We numerically compute $E(n,d)$ for $n\in\{2,3,4\}$, and show them
in Table \ref{table:Ed} (for $d\leq 10$) and Figure \ref{fig:Ed}b (for $d\leq 20$). We also plot the lower and upper bounds for $E(2,d)$ obtained in Theorem~\ref{thm: estimate for E(d)} and compare them with its numerical computation, see Figure \ref{fig:Ed}a. We note that for small $n$ and  $d$ (namely, $n\leq 5$ and $d\leq 4$), $E(n,d)$  can also be computed numerically via the probabilities $p_i$ of observing exactly $i$ equilibria using~\eqref{eq: definition E}. This direct approach has been used  in~\cite{HTG12} and~\cite{gokhale:2010pn}, which  our results are compatible with. Furthermore, as a consequence of Corollary 2,  we can derive   the exact  expected number of stable equilibria  for $n = 2$, which is equal to half of the expected number of equilibria given in the first row of Table \ref{table:Ed}.

\begin{table}[ht]
\caption{Expected number of internal equilibria  for some  values of $d$ and $n$. }
\vspace{0.2cm}
\begin{tabular}{l c l*{7}{c}r}
d                           & 2 & 3 & 4 & 5 & 6  & 7 & 8 & 9 & 10 \\
\hline
\hline
$n = 2$    &         0.5 & 0.77 & 0.98 & 1.16 & 1.31 & 1.47 & 1.60 & 1.73& 1.84 \\
$n = 3$            & 0.25 & 0.57 & 0.92 & 1.29 & 1.67 & 2.06 & 2.46 & 2.86 & 3.27 \\
$n = 4$            & 0.125 & 0.41 & 0.84 & 1.39 & 2.05 & 2.81  & 3.67 & 4.62 & 5.66  \\
\end{tabular}
\label{table:Ed}
 \end{table}

\begin{figure}
\centering
\includegraphics[width = 0.9\linewidth]{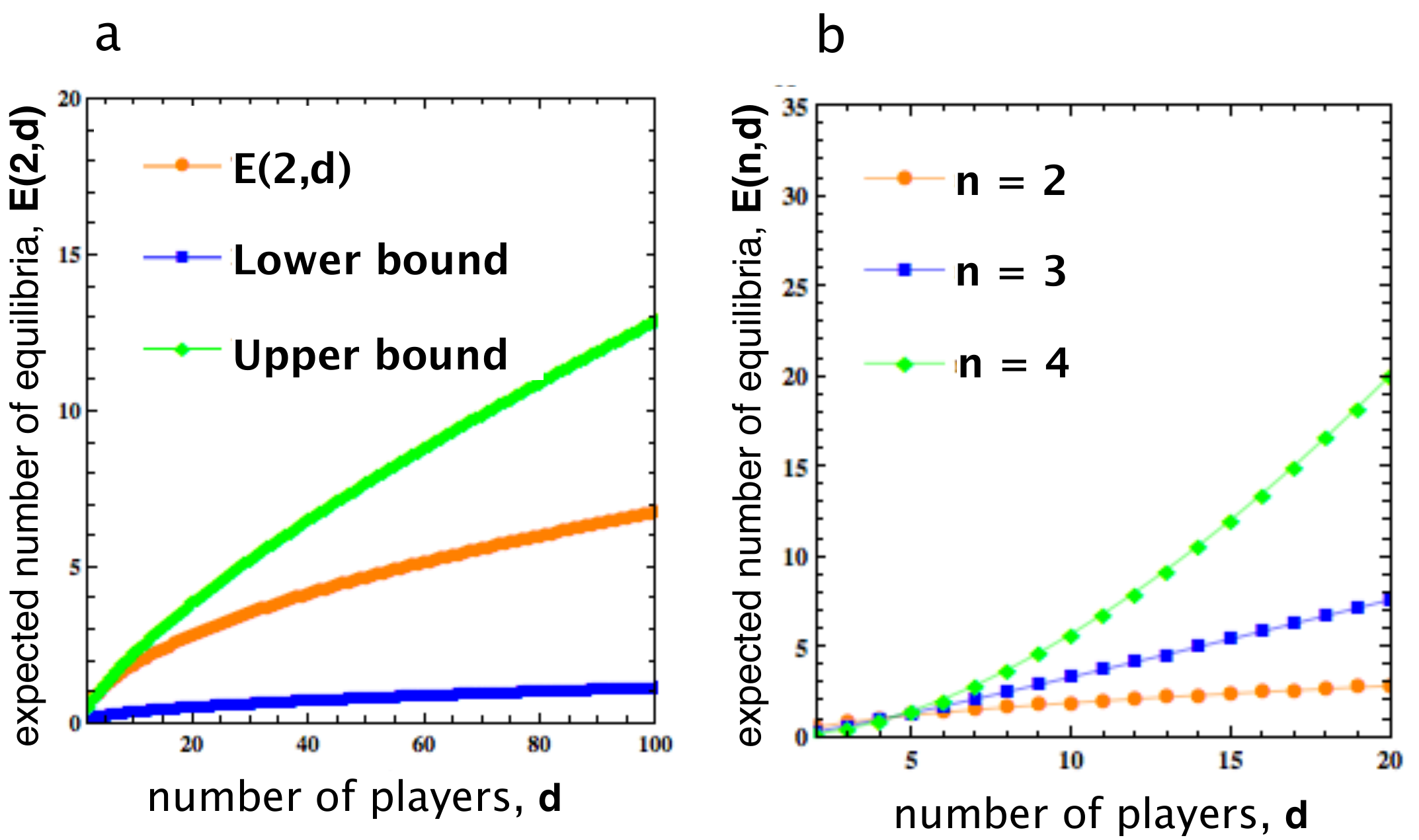}
\caption{\textbf{(a) Expected number of equilibria, $E(2,d)$, for varying the number of players in the game, and the upper and lower bounds obtained in Theorem~\ref{thm: estimate for E(d)}}. \textbf{(b) Expected number of equilibria, $E(n,d)$, for varying the number of players and strategies in the game}. We plot for different number of strategies: $n = 2$, 3 and 4. The payoff entries of the d-player n-strategy game are randomly drawn from the normal distribution.  In general, the larger the number of players in the game, the higher the  number of equilibria  one can expect if the payoff entries of the game are randomly chosen. We observe that for small $d$, $E(n,d)$ increases with $n$ while it decreases  for large $d$ (namely, $d \geq 5$, see also Table \ref{table:Ed}). The results were obtained  numerically  using Mathematica. }
\label{fig:Ed}
\end{figure} 

\section{Discussion}
\label{sec: discussion}
In the evolutionary game literature, mathematical results have been mostly provided for two-player games \cite{hofbauer:1998mm,nowak:2006bo}, despite the abundance of multi-player game examples in nature \cite{broom:1997aa,gokhale:2010pn}. The mathematical study for multi-player games has only attracted more attention recently, apparently because the resulting complexity and dynamics from the games are significantly magnified with the number of the game participants \cite{broom:1997aa,HTG12,wu2013dynamic}. As seen from our analysis, multi-player games introduce non-linearity into the inducing equations, because the fitness functions are polynomial instead of linear as in the pairwise interactions  \cite{hofbauer:1998mm,gokhale:2010pn,wu2013dynamic}. It is even more so if we consider scenarios in which the payoff matrix is non-linear naturally, such as for public goods games with non-linear returns \cite{pacheco:2009aa}, as observed in animals hunting \cite{stander1992cooperative,packer1997foraging}. In addition, as the number of strategies increases, one needs to deal with systems of multivariate polynomials. 
This seemingly complexity of the general multi-player multi-strategy games makes the analysis of the equilibrium points extremely difficult. For instance, in \cite{HTG12}, as the analysis was based on an  explicit calculation of the zeros of systems of polynomials, it cannot go beyond games with a small number of players (namely, $d \leq 5$), even for the simplest case of $n = 2$. 
Here we have taken a different approach based on the well-established theory of random polynomials. We have derived a computationally implementable formula, $E(n,d)$, of the mean number of equilibria in a general random $d$-player game with $n$ strategies. For $n = 2$ and with an arbitrary $d$, we have derived    asymptotic upper and lower bounds of $E(2,d)$. An interesting implication of these results is that although the expected number of equilibria tends to infinity when $d$ increases, the probability of seeing the maximal possible number of equilibria tends to 0. This is a notable observation since knowing the maximal number of equilibrium points in an evolutionary process is insightful and has been of special focus in biological contexts \cite{levin:2000aa,gokhale:2010pn}. Furthermore, for $d = 2$ with an arbitrary $n$, we have derived the exact value  $E(n,2) = 2^{1 - n}$, recovering results obtained in \cite{HTG12}. 
In the general case, based on the formula that we have derived, we have been able to numerically calculate $E(n,d)$, thereby going beyond the analytical and numerical computations from previous work that made use of the direct approach. 
As a consequence of the results obtained for the extreme cases, we could also derive the boundaries for the expected number of equilibrium points when  counting only the stable ones. Note that some studies have been carried out to study the stable equilibrium points or evolutionary stable strategies (ESS) \cite{maynard-smith:1982to} in a random evolutionary game, but it has been mostly done for two-player games, see for example  \cite{haigh1988distribution,kontogiannis2009support,hart2008evolutionarily}. 
As such, we have provided further understanding with respect to the  asymptotic behaviour of the expected number of stable equilibria for  multi-player random games, though with only two strategies, leaving the more general case for future study.

On the other hand, in the random polynomial literature, the study of distribution of zeros of system of random polynomials as described in~\eqref{eq: eqn for systems}, especially in the univariate case, has been carried out by many authors, see for instance~\cite{EK95} for a nice exposition and~\cite{TaoVu14} for the most recent results. The most well-known classes of polynomials are: flat polynomials or Weyl polynomials for $a_i\colonequals\frac{1}{i!}$, elliptic polynomials or binomial polynomials for $a_i\colonequals\sqrt{\begin{pmatrix}
d-1\\
i
\end{pmatrix}}$ and Kac polynomials for $a_i\colonequals 1$. We emphasize the difference between the polynomial studied in this paper with the elliptic case: $a_i=\begin{pmatrix}
d-1\\
i
\end{pmatrix}$ are binomial coefficients, not their square root. In the elliptic case, $v(x)^TCv(y)=\sum_{i=1}^{d-1}\begin{pmatrix}
d-1\\
i
\end{pmatrix}x^iy^i=(1+xy)^{d-1}$, and as a result $E=E(2,d)=\sqrt{d-1}$. While in our case, because of the square in the coefficients,
$v(x)^TCv(y)=\sum_{i=1}^{d-1}\begin{pmatrix}
d-1\\
i
\end{pmatrix}^2x^iy^i$ is no longer a generating function. Whether one can find an exact formula for $E(2,d)$ is unclear. For the multivariate situation, the exact formula for $E(n,2)$ is interesting by itself and we could not be able to find it in the literature. Due to the complexity in the general case $d,n\geq 3$, further research is required.   

In short, we have described a novel approach to calculating and analyzing the expected number of (stable) equilibrium points in a general random evolutionary game, giving insights into the overall complexity of such dynamical systems as the players and the strategies in the game increase. Since the theory of random polynomials is rich, we envisage that our method could be extended to obtain results for other more complex scenarios such as games having a payoff matrix  with dependent entries and/or with general distributions.   
\section{Appendix} 
\subsection{Properties of $a_k$ }
\label{sec: properties of a}
In the following we prove that $a_{2d-4-k} = a_k$ for all $0 \leq k \leq 2d-4$, i.e.~\eqref{eq: properties of a}. First, we transform $a_k$ as follows 
\begin{align*}
a_k &=\sum\limits_{\substack{0\leq i\leq d-1\\ 0\leq j\leq d-2\\i+j=k}}\begin{pmatrix}
d-1\\
i
\end{pmatrix}^2\begin{pmatrix}
d-2\\
j
\end{pmatrix}^2-\sum\limits_{\substack{0\leq i'\leq d-2\\ 0\leq j'\leq d-2\\i'+j'=k-1}}\begin{pmatrix}
d-2\\
i'
\end{pmatrix}\begin{pmatrix}
d-1\\
i'+1
\end{pmatrix}\begin{pmatrix}
d-2\\
j'
\end{pmatrix}\begin{pmatrix}
d-1\\
j'+1
\end{pmatrix} \\
 & =
 \begin{pmatrix} 
 d-2 \\
 k 
 \end{pmatrix}^2 + \sum\limits_{\substack{0\leq i\leq d-2\\ 0\leq j\leq d-2\\i+j=k-1}}\begin{pmatrix}
d-1\\
i+1
\end{pmatrix}^2\begin{pmatrix}
d-2\\
j
\end{pmatrix}^2-\sum\limits_{\substack{0\leq i'\leq d-2\\ 0\leq j'\leq d-2\\i'+j'=k-1}}\begin{pmatrix}
d-2\\
i'
\end{pmatrix}\begin{pmatrix}
d-1\\
i'+1
\end{pmatrix}\begin{pmatrix}
d-2\\
j'
\end{pmatrix}\begin{pmatrix}
d-1\\
j'+1
\end{pmatrix}\\
& =
 \begin{pmatrix} 
 d-2 \\
 k 
 \end{pmatrix}^2 + \sum\limits_{\substack{0\leq i\leq d-2\\ 0\leq j\leq d-2\\i+j=k-1}}\begin{pmatrix}
d-1\\
i+1
\end{pmatrix}\begin{pmatrix}
d-2\\
j
\end{pmatrix}\left(\begin{pmatrix}
d-1\\
i+1
\end{pmatrix}\begin{pmatrix}
d-2\\
j
\end{pmatrix} - 
\begin{pmatrix}
d-2\\
i
\end{pmatrix}\begin{pmatrix}
d-1\\
j+1
\end{pmatrix}
 \right) \\
 & =
 \begin{pmatrix} 
 d-2 \\
 k 
 \end{pmatrix}^2 + \sum\limits_{\substack{0\leq i \leq j \leq d-2\\i+j=k-1}}\left(\begin{pmatrix}
d-1\\
i+1
\end{pmatrix}\begin{pmatrix}
d-2\\
j
\end{pmatrix} - 
\begin{pmatrix}
d-2\\
i
\end{pmatrix}\begin{pmatrix}
d-1\\
j+1
\end{pmatrix}
 \right)^2
\end{align*}
Now we prove that $a_{2d-4-k} = a_{k}$. Indeed, we have 
\begin{align*}
a_{2d - 4 - k} &=
 \begin{pmatrix} 
 d-2 \\
 2d - 4 - k 
 \end{pmatrix}^2 + \sum\limits_{\substack{0\leq i \leq j \leq d-2\\i+j=2d-k-5}}\left(\begin{pmatrix}
d-2\\
i+1
\end{pmatrix}\begin{pmatrix}
d-2\\
j
\end{pmatrix} - 
\begin{pmatrix}
d-2\\
i
\end{pmatrix}\begin{pmatrix}
d-1\\
j+1
\end{pmatrix}
 \right)^2 \\
 &\text{(we use here the transformations $i = d-3 - i$ and $j = d-3 - j$)}\\
 &=
 \begin{pmatrix} 
 d-2 \\
 2d - 4 - k 
 \end{pmatrix}^2 + \sum\limits_{\substack{-1\leq j \leq i \leq d-3\\i+j=k-1}}\left(\begin{pmatrix}
d-1\\
i+1
\end{pmatrix}\begin{pmatrix}
d-2\\
j+1
\end{pmatrix} - 
\begin{pmatrix}
d-2\\
i+1
\end{pmatrix}\begin{pmatrix}
d-1\\
j+1
\end{pmatrix}
 \right)^2 \\
&=
 \begin{pmatrix} 
 d-2 \\
 2d - 4 - k 
 \end{pmatrix}^2 + \sum\limits_{\substack{-1\leq j \leq i \leq d-3\\i+j=k-1}}\left(\begin{pmatrix}
d-1\\
i+1
\end{pmatrix}\begin{pmatrix}
d-2\\
j
\end{pmatrix} - 
\begin{pmatrix}
d-2\\
i
\end{pmatrix}\begin{pmatrix}
d-1\\
j+1
\end{pmatrix}
 \right)^2. \\
\end{align*}
Since for $j = -1$ and $i  = k$ we have 
\[
\left(\begin{pmatrix}
d-1\\
i+1
\end{pmatrix}\begin{pmatrix}
d-2\\
j
\end{pmatrix} - 
\begin{pmatrix}
d-2\\
i
\end{pmatrix}\begin{pmatrix}
d-1\\
j+1
\end{pmatrix}
 \right)^2 =  \begin{pmatrix} 
 d-1 \\
 k 
 \end{pmatrix}^2,
\]
and for  $j = d-2$ and $i  = k-1 - (d-2)$, we have 
\begin{align*}
\left(\begin{pmatrix}
d-1\\
i+1
\end{pmatrix}\begin{pmatrix}
d-2\\
j
\end{pmatrix} - 
\begin{pmatrix}
d-2\\
i
\end{pmatrix}\begin{pmatrix}
d-1\\
j+1
\end{pmatrix}
 \right)^2 &= 
 \left(\begin{pmatrix}
d-1\\
k-(d-2)
\end{pmatrix}- 
\begin{pmatrix}
d-2\\
k-1-(d-2)
\end{pmatrix}
 \right)^2 \\
  &=
\begin{pmatrix}
d-2\\
k-(d-2)
\end{pmatrix}^2  \\
  &=
\begin{pmatrix}
d-2\\
2d-4-k
\end{pmatrix}^2  ,
\end{align*}
it follows that $a_{2d-4-k} = a_k$ for all $0 \leq k \leq 2d-4$.

\subsection{Detailed computation of $f(1)$}
\label{sec: f(1)}
We use the following identities involving the square of binomial coefficients.
\begin{align*}
&M_d(1)=\sum\limits_{k=0}^{d-1}\begin{pmatrix}
d-1\\
k
\end{pmatrix}^2=\begin{pmatrix}
2(d-1)\\
d-1
\end{pmatrix},
\\& A_d(1)=(d-1)^2M_{d-1}(1)=(d-1)^2\begin{pmatrix}
2(d-2)\\
d-2
\end{pmatrix},
\\& B_d(1)=\sum\limits_{k=1}^{d-1}k\begin{pmatrix}
d-1\\
k
\end{pmatrix}^2=\frac{d-1}{2}\begin{pmatrix}
2(d-1)\\
d-1
\end{pmatrix}.
\end{align*}
Therefore, we have
\begin{align*}
f(1)&=\frac{1}{\pi}\frac{\sqrt{A_d(1)M_d(1)-B_d(1)^2}}{M_d(1)}
\\&=\frac{1}{\pi}\frac{\sqrt{(d-1)^2\begin{pmatrix}
2(d-1)\\
d-1
\end{pmatrix}\left[\begin{pmatrix}
2(d-2)\\
d-2
\end{pmatrix}-\frac{1}{4}\begin{pmatrix}
2(d-1)\\
d-1
\end{pmatrix}\right]}}{\begin{pmatrix}
2(d-1)\\
d-1
\end{pmatrix}}
\\&=\frac{d-1}{\pi}\times\sqrt{\frac{\begin{pmatrix}
2(d-2)\\
d-2
\end{pmatrix}}{\begin{pmatrix}
2(d-1)\\
d-1
\end{pmatrix}}-\frac{1}{4}}
\\&=\frac{d-1}{\pi}\times\sqrt{\frac{d-1}{2(2d-3)}-\frac{1}{4}}
\\&=\frac{d-1}{2\pi\sqrt{2d-3}},
\end{align*}
where we have used the identity
$
\begin{pmatrix}
2(n+1)\\
n+1
\end{pmatrix}=\frac{n+1}{2(2n+1)}\begin{pmatrix}
2n\\
n
\end{pmatrix}.
$
\subsection{Alternative proof of the fifth property in Lemma~\ref{lemma: density of zeros}}
\label{sec: alternative proof}
We show here an alternative proof of the fifth property without using the first one in Lemma~\ref{lemma: density of zeros}.

Since $M_d\left(\frac{1}{t}\right)=t^{2(1-d)M_d(t)}$, we have
\begin{align*}
M_d'\left(\frac{1}{t}\right)&=-t^{3-2d}\left[2(1-d)M_d(t)+tM_d'(t)\right],
\\M_d''\left(\frac{1}{t}\right)&=t^{4-2d}\left[2(3-2d)(1-d)M_d(t)+2(3-2d)tM_d'(t)+t^2M_d''(t)\right].
\end{align*}
Therefore, from~\eqref{eq: A and B interms of M}, we have
\belowdisplayskip=0pt
\begin{equation*}
B_d\left(\frac{1}{t}\right)=\frac{1}{2}M_d'\left(\frac{1}{t}\right)=-t^{3-2d}\left[(1-d)M_d(t)+\frac{1}{2}tM_d'(t)\right],
\end{equation*}
and
\begin{align*}
A_d\left(\frac{1}{t}\right)&=\frac{t}{4}\left[M_d'\left(\frac{1}{t}\right)+\frac{1}{t}M_d''\left(\frac{1}{t}\right)\right]
\\&=\frac{1}{4}t^{4-2d}\left[4(1-d)^2+(5-4d)tM_d'(t)+t^2M_d''(t)\right].
\end{align*}
Hence
\begin{align*}
A_d\left(\frac{1}{t}\right)M_d\left(\frac{1}{t}\right)-B_d\left(\frac{1}{t}\right)^2&=\frac{1}{4}t^{6-2d}\left[4(1-d)^2M_d(t)^2+(5-4d)tM_d'(t)M_d(t)+t^2M_d''(t)M_d(t)\right]
\\&\qquad -t^{6-4d}\left[(1-d)M_d(t)+\frac{1}{2}tM_d'(t)\right]^2
\\&=\frac{1}{4}t^{6-2d}\left[tM_d(t)M_d'(t)+t^2M_d(t)M_d''(t)-t^2M_d'(t)^2\right]
\\&=\frac{1}{4}t^{8-2d}\left[\frac{1}{t}M_d(t)M_d'(t)+M_d(t)M_d''(t)-M_d'(t)^2\right].
\end{align*}
Therefore,
\begin{align*}
f\left(\frac{1}{t}\right)&=\frac{1}{\pi}\frac{\sqrt{A_d\left(\frac{1}{t}\right)M_d\left(\frac{1}{t}\right)-B_d\left(\frac{1}{t}\right)^2}}{M_d\left(\frac{1}{t}\right)}\\
\\&=\frac{1}{\pi}\frac{t^{4-2d}\sqrt{\frac{1}{4t}\left[M_d(t)M_d'(t)+tM_d(t)M_d''(t)-tM_d'(t)^2\right]}}{t^{2-2d}M_d(t)}
\\&=t^2\frac{1}{\pi}\frac{\sqrt{\frac{1}{4t}\left[M_d(t)M_d'(t)+tM_d(t)M_d''(t)-tM_d'(t)^2\right]}}{M_d(t)}
\\&=t^2f(t).
\end{align*}
\subsection{Proof of Lemma~\ref{lem: det L}}
\label{sec: computeDetL}
Denoting $\Sigma = 1+\sum_{k=1}^{n-1} t_k^2$, we have 
\vskip5mm
\begin{align*}
\det L &= \frac{1}{\Sigma^{2(n-1)}} 
\det \begin{pmatrix}
\Sigma-t_1^2 & -t_1 t_2& \dots & -t_1 t_{n-1} \\
 -t_2 t_1& \Sigma-t_2^2 & \dots & -t_2 t_{n-1} \\
\dots & \dots & \dots  & \dots \\
-t_{n-1} t_1&  -t_{n-1} t_2 & \dots &\Sigma-t_{n-1}^2
\end{pmatrix}
 \\&= \frac{1}{t_1 \dots t_{n-1}} \frac{1}{\Sigma^{2(n-1)}} \det \begin{pmatrix}
t_1(\Sigma-t_1^2) & -t_1 t_2^2& \dots & -t_1 t_{n-1}^2 \\
 -t_2 t_1^2& t_2( \Sigma-t_2^2) & \dots & -t_2 t_{n-1}^2 \\
\dots & \dots & \dots  & \dots \\
-t_{n-1} t_1^2&  -t_{n-1} t_2^2 & \dots &t_{n-1}(\Sigma-t_{n-1}^2) 
 \end{pmatrix}
\\&= \frac{1}{t_1 \dots t_{n-1}} \frac{1}{\Sigma^{2(n-1)}} \det \begin{pmatrix}
t_1 & -t_1 t_2^2& \dots & -t_1 t_{n-1}^2 \\
t_2 & t_2( \Sigma-t_2^2) & \dots & -t_2 t_{n-1}^2 \\
\dots & \dots & \dots  & \dots \\
t_{n-1} &  -t_{n-1} t_2^2 & \dots &t_{n-1}(\Sigma-t_{n-1}^2)
\end{pmatrix}
\\&=  \frac{1}{\Sigma^{2(n-1)}}\det  \begin{pmatrix}
1 & - t_2^2& \dots & - t_{n-1}^2 \\
1 &  \Sigma-t_2^2 & \dots & - t_{n-1}^2 \\
\dots & \dots & \dots  & \dots \\
1 &  - t_2^2 & \dots &\Sigma-t_{n-1}^2 
\end{pmatrix}
\\&=\frac{1}{\Sigma^{2(n-1)}} \det \begin{pmatrix}
1 & - t_2^2& \dots & - t_{n-1}^2 \\
0 &  \Sigma & \dots & 0 \\
\dots & \dots & \dots  & \dots \\
0 &  0 & \dots &\Sigma
\end{pmatrix}
\\&=  \frac{1}{\Sigma^n}.
\end{align*}
\paragraph{Acknowledgment}We would like to thank members of the WMS seminar at the CASA at Eindhoven University of Technology, especially Dr. R. van Hassel and Prof. A. Blokhuis for useful discussion. Part of this work was done under the support of   the F.W.O. Belgium (TAH, postdoctoral fellowship 05\_05 TR 7142).

\bibliographystyle{alpha}
\bibliography{GTbib}

\end{document}